\documentclass{amsart}%
\usepackage{amsmath}
\usepackage{amssymb}
\usepackage{amsfonts}
\usepackage{graphicx}%
\newtheorem{theorem}{Theorem}[section]
\theoremstyle{plain}

\newtheorem{corollary}[theorem]{Corollary}

\newtheorem{definition}[theorem]{Definition}
\newtheorem{example}[theorem]{Example}

\newtheorem{lemma}[theorem]{Lemma}

\newtheorem{proposition}[theorem]{Proposition}

\numberwithin{equation}{section}
\begin{document}
\title[Elliptic problems with supercritical exponents]{Nonexistence and multiplicity of solutions to elliptic problems with
supercritical exponents}
\author{M\'{o}nica Clapp}
\address{Instituto de Matem\'{a}ticas, Universidad Nacional Aut\'{o}noma de M\'{e}xico,
Circuito Exterior, C.U., 04510 M\'{e}xico D.F., Mexico}
\email{mclapp@matem.unam.mx}
\author{Jorge Faya}
\address{Instituto de Matem\'{a}ticas, Universidad Nacional Aut\'{o}noma de M\'{e}xico,
Circuito Exterior, C.U., 04510 M\'{e}xico D.F., Mexico}
\email{jorgefaya@gmail.com}
\author{Angela Pistoia}
\address{Dipartimento di Metodi e Modelli Matematici, Universit\'{a} di Roma "La
Sapienza", via Antonio Scarpa 16, 00161 Roma, Italy}
\email{pistoia@dmmm.uniroma1.it}
\thanks{M. Clapp and J. Faya are supported by CONACYT grant 129847 and PAPIIT grant
IN106612 (Mexico). A. Pistoia is supported by Universit\`{a} degli Studi di
Roma "La Sapienza" Accordi Bilaterali "Esistenza e propriet\`{a} geometriche
di soluzioni di equazioni ellittiche non lineari" (Italy).}
\date{September 2012}
\maketitle

\begin{abstract}
We consider the supercritical problem%
\[
-\Delta u=\left\vert u\right\vert ^{p-2}u\text{ \ in }\Omega,\quad u=0\text{
\ on }\partial\Omega,
\]
where $\Omega$ is a bounded smooth domain in $\mathbb{R}^{N},$ $N\geq3,$ and
$p\geq2^{\ast}:=\frac{2N}{N-2}.$

Bahri and Coron showed that if $\Omega$ has nontrivial homology this problem
has a positive solution for $p=2^{\ast}.$ However, this is not enough to
guarantee existence in the supercritical case. For $p\geq\frac{2(N-1)}{N-3}$
Passaseo exhibited domains carrying one nontrivial homology class in which no
nontrivial solution exists. Here we give examples of domains whose homology
becomes richer as $p$ increases. More precisely, we show that for
$p>\frac{2(N-k)}{N-k-2}$ with $1\leq k\leq N-3$ there are bounded smooth
domains in $\mathbb{R}^{N}$ whose cup-length is $k+1$ in which this problem
does not have a nontrivial solution.

For $N=4,8,16$ we show that there are many domains, arising from the Hopf
fibrations, in which the problem has a prescribed number of solutions for some
particular supercritical exponents.\medskip



\end{abstract}

\section{Introduction}

We consider the problem
\begin{equation}
\left\{
\begin{array}
[c]{ll}%
-\Delta u=\left\vert u\right\vert ^{p-2}u & \text{in }\Omega,\\
\hspace{0.6cm}u=0 & \text{on }\partial\Omega,
\end{array}
\right.  \tag{$\wp_p$}\label{prob}%
\end{equation}
where $\Omega$ is a bounded smooth domain in $\mathbb{R}^{N},$ $N\geq3,$ and
$p\geq2^{\ast},$ where $2^{\ast}:=\frac{2N}{N-2}$ is the critical Sobolev exponent.

It is well known that the existence of a solution depends on the domain.
Pohozhaev's identity \cite{po} implies that (\ref{prob}) does not have a
nontrivial solution if $\Omega$ is strictly starshaped. On the other hand,
Kazdan and Warner \cite{kw} showed that infinitely many radial solutions exist
if $\Omega$ is an annulus.

For $p=2^{\ast}$ a remarkable result obtained by Bahri and Coron \cite{bc}
establishes the existence of at least one positive solution to problem
$(\wp_{2^{\ast}})$ in every domain $\Omega$ having nontrivial reduced homology
with $\mathbb{Z}/2$-coefficients. Multiplicity results are also available,
either for domains which are small perturbations of a given one, as in
\cite{gmp}, or for domains which have enough, but possibly finite, symmetries,
as in \cite{cf}. A more detailed discussion may be found in these papers.

Unlike the critical case, in the supercritical case the existence of a
nontrivial cohomology class in $\Omega$ does not guarantee the existence of a
nontrivial solution to problem (\ref{prob}). In fact, for each $1\leq k\leq
N-3,$ Passaseo \cite{pa1,pa2} exhibited domains having the homotopy type of a
$k$-dimensional sphere $\mathbb{S}^{k}$ in which problem (\ref{prob}) does not
have a nontrivial solution for any $p\geq2_{N,k}^{\ast}:=\frac{2(N-k)}%
{N-k-2}.$ We call $2_{N,k}^{\ast}$ the $(k+1)$-st critical exponent. It is the
critical exponent for the Sobolev embedding $H^{1}(\mathbb{R}^{N-k}%
)\hookrightarrow L^{q}(\mathbb{R}^{N-k}).$ Nonexistence of bounded positive
solutions for $p>2_{N,k}^{\ast}$ in a thin enough tubular neighborhood of a
$k$-dimensional submanifold of $\mathbb{R}^{N}$\ was recently shown in
\cite{pps}.

The first nontrivial existence result for $p>2^{\ast}$ was obtained by del
Pino, Felmer and Musso \cite{dfm} in the slightly supercritical case, i.e. for
$p>2^{\ast}$ but close enough to $2^{\ast}.$ This case was also considered in
\cite{mpa,pr} where multiplicity was established. In \cite{dw} existence was
established in a domain with a small enough hole for a.e. $p>2^{\ast}$,
whereas in \cite{bcgp, pps} solutions of a particular type were constructed in
a tubular neighborhood of fixed radius of an expanding manifold for every $p.$
The problem for $p$ slightly below the second critical exponent was considered
in \cite{dmp} where solutions for $p=2_{N,1}^{\ast}-\varepsilon$ concentrating
at a boundary geodesic as $\varepsilon\rightarrow0$ have been constructed in
certain domains. Quite recently, positive and sign changing solutions for
$p=2_{N,k}^{\ast}-\varepsilon$ which concentrate at $k$-dimensional
submanifolds of the boundary as $\varepsilon\rightarrow0$ were exhibited in
\cite{acp}, while in \cite{kp} positive and sign changing solutions for $p$
large which concentrate at $(N-2)$-dimensional submanifolds of the boundary as
$p\rightarrow+\infty$ have been constructed.

In a recent work Wei and Yan \cite{wy} exhibited domains $\Omega$ in which
problem (\ref{prob})\ has infinitely many positive solutions for
$p=2_{N,k}^{\ast}$. They considered domains $\Omega$ of the form
\begin{equation}
\Omega:=\{(y,z)\in\mathbb{R}^{k+1}\times\mathbb{R}^{N-k-1}:\left(  \left\vert
y\right\vert ,z\right)  \in\Theta\}, \label{rotOmega}%
\end{equation}
where $\Theta$ is a bounded smooth domain in $\mathbb{R}^{N-k}$ with
$\overline{\Theta}\subset\left(  0,\infty\right)  \times\mathbb{R}^{N-k-1}$
which satisfies certain geometric assumptions.

For domains of this type we give a geometric condition which guarantees nonexistence.

\begin{definition}
\label{**}We shall say that $\Theta$ is doubly starshaped with respect to
$\mathbb{R}\times\left\{  0\right\}  $ if there exist two numbers
$0<t_{0}<t_{1}$ such that $t\in(t_{0},t_{1})$ for every $(t,z)\in$ $\Theta$
and $\Theta$ is strictly starshaped with respect to $\xi_{0}:=(t_{0},0)$ and
to $\xi_{1}:=(t_{1},0)$, i.e.
\[
\left\langle x-\xi_{i},\nu_{\Theta}(x)\right\rangle >0\qquad\forall
x\in\partial\Theta\smallsetminus\left\{  \xi_{i}\right\}  ,
\]
for each $i=0,1,$ where $\nu_{\Theta}(x)$ is the outward pointing unit normal
to $\partial\Theta$ at $x.$
\end{definition}

For $\Omega$ as in (\ref{rotOmega}) and $K\in\mathcal{C}^{1}(\overline{\Omega
})$ we consider the problem%
\begin{equation}
\left\{
\begin{array}
[c]{ll}%
-\Delta u=K(y,z)\left\vert u\right\vert ^{p-2}u & \text{in }\Omega,\\
\hspace{0.6cm}u=0 & \text{on }\partial\Omega.
\end{array}
\right.  \label{probK}%
\end{equation}
We assume $K$ to be strictly positive on $\overline{\Omega}$ and radially
symmetric in $y,$ i.e. $K(y,z)=K(\left\vert y\right\vert ,z).$ We prove the
following result.

\begin{theorem}
\label{thmmain1}If $\Theta$ is doubly starshaped with respect to
$\mathbb{R}\times\left\{  0\right\}  \ $and if $\left\langle y,\partial
_{y}K(y,z)\right\rangle \leq0$ and $\left\langle z,\partial_{z}%
K(y,z)\right\rangle \leq0$ for all $(y,z)\in\Omega,$ then problem
\emph{(\ref{probK})} does not have a nontrivial solution for $p\geq
2_{N,k}^{\ast}$ and has infinitely many solutions for $p\in(2,2_{N,k}^{\ast
}),$ where $0\leq k\leq N-3.$
\end{theorem}

The domains in Passaseo's examples \cite{pa1,pa2} are defined as in
(\ref{rotOmega}) with $\Theta$ being a ball centered at some point $(\tau,0),$
which is obviously doubly starshaped with respect to $\mathbb{R}\times\left\{
0\right\}  .$ We stress that it is not enough for $\Theta$ to be strictly
starshaped to guarantee nonexistence: the domains considered by Wei and Yan
\cite{wy} are obtained from a domain $\Theta$ which is not doubly starshaped
with respect to $\mathbb{R}\times\left\{  0\right\}  $, but which may be
chosen to be strictly starshaped.

The domains in Passaseo's examples \cite{pa1,pa2}, as well as those in Theorem
\ref{thmmain1},\ have the homotopy type of $\mathbb{S}^{k}.$ One may ask
whether there are examples of domains having a richer topology for which a
similar nonexistence result holds true. We prove the following result.

\begin{theorem}
\label{thmmain4}Given $k=k_{1}+\cdots+k_{m}$ with $k_{i}\in\mathbb{N}$ and
$k\leq N-3,$ and $\varepsilon>0$ there exists a bounded smooth domain $\Omega$
in $\mathbb{R}^{N},$ which has the homotopy type of $\mathbb{S}^{k_{1}}%
\times\cdots\times\mathbb{S}^{k_{m}},$ in which problem \emph{(\ref{prob})}
does not have a nontrivial solution for $p\geq2_{N,k}^{\ast}+\varepsilon$ and
has infinitely many solutions for $p\in(2,2_{N,k}^{\ast}).$
\end{theorem}

In particular, if we take all $k_{i}=1,$ the domain $\Omega$ is homotopy
equivalent to the product of $k$ circles. So not only the homology of $\Omega$
is nontrivial but there are $k$ different cohomology classes in $H^{1}%
(\Omega;\mathbb{Z})$ whose cup-product is the generator of $H^{k}%
(\Omega;\mathbb{Z})$. Hence, the cup-length of $\Omega$ equals $k+1.$

We also obtain an existence result for a different type of domains, arising
from the Hopf fibrations. We are specifically interested in the cases where
$N=4,8,16.$ In these cases $\mathbb{R}^{N}\mathbb{=K}\times\mathbb{K}$, where
$\mathbb{K}$ is either the complex numbers $\mathbb{C}$, the quaternions
$\mathbb{H}$ or the Cayley numbers $\mathbb{O}.$ The set of units
$\mathbb{S}_{\mathbb{K}}:=\{\zeta\in\mathbb{K}:\left\vert \zeta\right\vert
=1\},$ which is a group if $\mathbb{K=C}$ or $\mathbb{H}$ and a quasigroup
with unit if $\mathbb{K=O}$, acts on $\mathbb{R}^{N}$ by multiplication on
each coordinate, i.e. $\zeta(z_{1},z_{2}):=(\zeta z_{1},\zeta z_{2}).$ The
orbit space of $\mathbb{R}^{N}$ with respect to this action turns out to be
$\mathbb{R}^{\dim\mathbb{K}+1}$ and the projection onto the orbit space is the
Hopf map $\pi:\mathbb{R}^{N}=\mathbb{K}\times\mathbb{K}\rightarrow
\mathbb{K}\times\mathbb{R}=\mathbb{R}^{\dim\mathbb{K}+1}$ given by%
\[
\pi(z_{1},z_{2}):=(2\overline{z_{1}}z_{2},\,\left\vert z_{1}\right\vert
^{2}-\left\vert z_{2}\right\vert ^{2})\text{.}%
\]
We consider domains of the form $\Omega=\pi^{-1}(U)$ where $U$ is a bounded
smooth domain in $\mathbb{R}^{\dim\mathbb{K}+1}.$ We assume that $U$ is
invariant under the action of some closed subgroup $G$ of the group
$O(\dim\mathbb{K}+1)$ of linear isometries of $\mathbb{R}^{\dim\mathbb{K}+1}.$
We denote by $Gx:=\{gx:g\in G\}$ the $G$-orbit of a point $x\in\mathbb{R}%
^{\dim\mathbb{K}+1}$ and by $\#Gx$ its cardinality. Recall that $U $ is called
$G$-invariant if $Gx\subset U$ for all $x\in U,$ and a function
$u:U\rightarrow\mathbb{R}$ is called $G$-invariant if $u$ is constant on every
$Gx.$

Fix a closed subgroup $\Gamma$ of $O(\dim\mathbb{K}+1)$ and a nonempty
$\Gamma$-invariant bounded smooth domain $D$ in $\mathbb{R}^{\dim\mathbb{K}%
+1}$ such that $\#\Gamma x=\infty\ $for all $x\in D.$ We prove the following result.

\begin{theorem}
\label{thmmain2}There exists an increasing sequence $(\ell_{m})$ of positive
real numbers, depending only on $\Gamma$ and $D$, with the following property:
If $U$ contains $D$ and if it is invariant under the action of a closed
subgroup $G$ of $\Gamma$ for which%
\[
\min_{x\in U}\left(  \#Gx\right)  \left\vert x\right\vert ^{\frac
{\dim\mathbb{K}-1}{2}}>\ell_{m}%
\]
holds, then, for $p=2_{N,\dim\mathbb{K}-1}^{\ast},$ problem \emph{(\ref{prob}%
)} has at least $m$ pairs of\ solutions $\pm u_{1},\ldots,\pm u_{m}$ in
$\Omega:=\pi^{-1}(U),$ which are constant on $\pi^{-1}(Gx)$ for each $x\in U.$
In particular, they are $\mathbb{S}_{\mathbb{K}}$-invariant. Moreover, $u_{1}$
is positive and $u_{2},\ldots,u_{m}$ change sign.
\end{theorem}

For example, we may fix a bounded smooth domain $D_{0}$ in $\mathbb{R}^{2}$
with $\overline{D_{0}}\subset(0,\infty)\times\mathbb{R}$ and set
\[
D:=\{(z,t)\in\mathbb{K}\times\mathbb{R}:(\left\vert z\right\vert ,t)\in
D_{0}\}.
\]
Then $D$ is invariant under the action of the group $\Gamma:=\mathbb{S}%
_{\mathbb{C}}$ of unit complex numbers on $\mathbb{K}\times\mathbb{R}$ given
by $e^{i\theta}(z,t):=(e^{i\theta}z,t)$. If $G_{n}:=\{e^{2\pi ik/n}%
:k=0,...,n-1\}$ is the cyclic subgroup of $\Gamma$ of order $n,$ then
$\#G_{n}x=n$ for every $x\in(\mathbb{K}\smallsetminus\{0\}\mathbb{)}%
\times\mathbb{R}.$ Therefore, for every $G_{n}$-invariant bounded smooth
domain $U$ in $\mathbb{K}\times\mathbb{R}$ with
\[
D\subset U\subset(\mathbb{K}\smallsetminus\{0\}\mathbb{)}\times\mathbb{R}%
\text{\quad and\quad}n\left\vert x\right\vert ^{\frac{\dim\mathbb{K}-1}{2}%
}>\ell_{m},
\]
Theorem \ref{thmmain2} yields at least $m$ pairs of solutions to problem
(\ref{prob}) in $\Omega:=\pi^{-1}(U)$ for $p=2_{N,\dim\mathbb{K}-1}^{\ast}.$

In contrast to \cite{wy}, where multiplicity is established\ using
Lyapunov-Schmidt reduction, the proof of the Theorem \ref{thmmain2} uses
variational methods. It is based on the following result.

\begin{proposition}
\label{propHopf}Let $N=2,4,8,16$, $\ U$ be a bounded smooth domain in
$\mathbb{R}^{\dim\mathbb{K}+1}$ which does not contain the origin,
$a\in\mathbb{R}$, and $f:\mathbb{R}\rightarrow\mathbb{R}$. If $v$ solves%
\begin{equation}
\left\{
\begin{array}
[c]{ll}%
-\Delta v+\frac{a}{2\left\vert x\right\vert }v=\frac{1}{2\left\vert
x\right\vert }f(v) & \text{in }U,\\
\qquad\qquad\quad v=0 & \text{on }U,
\end{array}
\right.  \label{prob2}%
\end{equation}
then $u:=v\circ\pi$ is a solution of%
\begin{equation}
\left\{
\begin{array}
[c]{ll}%
-\Delta u+au=f(u) & \text{in }\Omega:=\pi^{-1}(U),\\
\qquad\qquad u=0 & \text{on }\partial\Omega,
\end{array}
\right.  \label{prob3}%
\end{equation}
where $\pi:\mathbb{R}^{N}\rightarrow\mathbb{R}^{\dim\mathbb{K}+1}$ is the Hopf
map. Conversely, if $u$ is an $\mathbb{S}_{\mathbb{K}}$-invariant solution of
\emph{(\ref{prob3})} and $u=v\circ\pi,$ then $v$ solves \emph{(\ref{prob2})}.
\end{proposition}

For $N=4$ this result was proved by Ruf and Srikanth in \cite{sr}\ by direct
computation. Here we derive it from the theory of harmonic morphisms (see
section \ref{sec:hm}).

Theorem \ref{thmmain4} does not apply to the case $p\in\lbrack2_{N,k}^{\ast
},2_{N,k}^{\ast}+\varepsilon)$. So the question remains open whether there
are\ examples of domains having the homotopy type of a product of spheres for
which nonexistence holds true for all $p\geq2_{N,k}^{\ast}$. We give a partial
answer as follows.

\begin{theorem}
\label{thmmain3}Let $N=4,8,16.$ Then there exist bounded smooth domains
$\Omega_{n}$ in $\mathbb{R}^{N}=\mathbb{K}\times\mathbb{K},$ which have the
homotopy type of $\mathbb{S}^{\frac{N-2}{2}}\times\mathbb{S}^{n}$ if $1\leq
n\leq\frac{N-4}{2}$ and of $\mathbb{S}^{\frac{N-2}{2}}$ if $n=0$, such that
problem \emph{(\ref{prob})} does not have a nontrivial $\mathbb{S}%
_{\mathbb{K}}$-invariant solution for $p\geq2_{N,k}^{\ast}$ and has infinitely
many $\mathbb{S}_{\mathbb{K}}$-invariant solutions for $p<2_{N,k}^{\ast}$
where $k:=\frac{N-2}{2}+n.$
\end{theorem}

The question remains open as to whether for such domains other solutions
exist, which are not $\mathbb{S}_{\mathbb{K}}$-invariant, particularly for
$p\geq2_{N,k}^{\ast}$.

This paper is organized as follows: in Section \ref{sec:hm} we present the
basic notions and results of the theory of harmonic morphisms and prove
Proposition \ref{propHopf}. Section \ref{sec:existence} is devoted to proving
Theorem \ref{thmmain2}. Theorems \ref{thmmain1}, \ref{thmmain4} and
\ref{thmmain3} are proved in Section \ref{sec:nonexistence}.

\section{Harmonic morphisms}

\label{sec:hm}We recall some basic notions and give examples of harmonic
morphisms. A detailed discusion is given e.g. in \cite{bw, er, w}.

Let $(M,\mathfrak{g})$ and $(N,\mathfrak{h})$ be Riemannian manifolds of
dimensions $m$ and $n$ respectively. A smooth map $\pi:M\rightarrow N$ is
called \emph{horizontally weakly conformal} if for each $x\in M$ at which
\textrm{d}$\pi_{x}\neq0$ the differential \textrm{d}$\pi_{x}:T_{x}M\rightarrow
T_{\pi(x)}N$ is surjective and horizontally conformal, i.e. there exists a
number $\lambda(x)\neq0$ such that
\[
\mathfrak{h}(\mathrm{d}\pi_{x}X,\mathrm{d}\pi_{x}Y)=\lambda^{2}(x)\mathfrak{g}%
(X,Y)\text{\qquad for all }X,Y\in T_{x}^{H}M\text{,}%
\]
where $T_{x}^{H}M$ denotes the orthogonal complement of $\ker\left(
\mathrm{d}\pi_{x}\right)  .$ Defining $\lambda(x)=0$ if \textrm{d}$\pi_{x}=0$
we obtain a function $\lambda:M\rightarrow\lbrack0,\infty)$ called the
\emph{dilation of }$\pi.$ It is given by $\lambda^{2}(x)=\frac{\left\vert
\mathrm{d}\pi_{x}\right\vert ^{2}}{n},$ where $\left\vert \mathrm{d}\pi
_{x}\right\vert $ is the Hilbert-Schmidt norm of \textrm{d}$\pi_{x}.$ Hence,
it is a smooth function.

If $\pi$ has no critical points (i.e. \textrm{d}$\pi_{x}\neq0$ for all $x\in
M$) then it is called a \emph{conformal submersion}. If $\lambda\equiv1$ then
$\pi:M\rightarrow N$ is a \emph{Riemannian submersion}. Note that, if the
dilation is constant and non-zero, then $\pi$ is a Riemannian submersion up to
scale, i.e. it is a Riemannian submersion after a suitable homothetic change
of metric on the domain or codomain.

The \emph{tension field} $\tau(\pi)$ of a smooth map $\pi:M\rightarrow N$ is
defined as%
\[
\tau(\pi):=\text{Trace}_{\mathfrak{g}}\nabla\mathrm{d}\pi.
\]
Thus, $\tau(\pi)$ is a vector field along $\pi,$ i.e. a section of the
pullback bundle $\pi^{-1}TN.$ In charts,%
\[
\tau(\pi)=\mathfrak{g}^{ij}(\nabla_{\partial_{i}}\mathrm{d}\pi)(\partial
_{j}),
\]
that is,%
\begin{align*}
\tau^{\gamma}(\pi)  &  =\mathfrak{g}^{ij}(\nabla\mathrm{d}\pi^{\gamma}%
)_{ij}+\mathfrak{g}^{ij}\Gamma_{\alpha\beta}^{N\,\gamma}\pi_{i}^{\alpha}%
\pi_{j}^{\beta}\\
&  =\mathfrak{g}^{ij}\left(  \frac{\partial^{2}\pi^{\gamma}}{\partial
x^{i}\partial x^{j}}-\Gamma_{ij}^{M\,k}\frac{\partial\pi^{\gamma}}{\partial
x^{k}}+\Gamma_{\alpha\beta}^{N\,\gamma}\pi_{i}^{\alpha}\pi_{j}^{\beta}\right)
\\
&  =-\Delta_{M}\pi^{\gamma}+\mathfrak{g}^{ij}\Gamma_{\alpha\beta}^{N\,\gamma
}\pi_{i}^{\alpha}\pi_{j}^{\beta},\text{\qquad}1\leq\gamma\leq n,
\end{align*}
where $\Delta_{M}$ is the Laplace-Bertrami operator on $M$ (with the customary
sign convention of Riemannian geometry) and $\Gamma_{ij}^{M\,k}$ and
$\Gamma_{\alpha\beta}^{N\,\gamma}$ are the Christoffel symbols of\ $M $ and
$N$ respectively. The map $\pi:M\rightarrow N$ is called \emph{harmonic} if
$\tau(\pi)\equiv0.$ If, in addition, $\pi$ is horizontally weakly conformal,
then $\pi$ is called a \emph{harmonic morphism.} The main property of harmonic
morphisms is the following one.

\begin{proposition}
\label{prophm}A smooth map $\pi:M\rightarrow N$ is a harmonic morphism with
dilation $\lambda$ iff%
\[
\Delta_{M}(v\circ\pi)=\lambda^{2}\left[  (\Delta_{N}v)\circ\pi\right]
\]
for each smooth function $v:V\rightarrow\mathbb{R}$ defined on an open subset
$V$ of $N$ with $\pi^{-1}(V)\neq\emptyset.$
\end{proposition}

\begin{proof}
See \cite[Proposition 4.2.3]{bw}.
\end{proof}

\begin{corollary}
\label{correduction}Let $\pi:M\rightarrow N$ be a harmonic morphism with
dilation $\lambda$, $\ a:V\rightarrow\mathbb{R}$ be a function defined on an
open subset $V$ of $N$ with $\pi^{-1}(V)\neq\emptyset,$ and $f:\mathbb{R}%
\rightarrow\mathbb{R}$. Assume there exists $\mu:V\rightarrow(0,\infty) $ such
that $\mu\circ\pi=\lambda^{2}$ on $\pi^{-1}(V).$ If $v:V\rightarrow\mathbb{R}$
solves%
\begin{equation}
\Delta_{N}v+\frac{a(y)}{\mu(y)}v=\frac{1}{\mu(y)}f(v), \label{eqN}%
\end{equation}
then $u:=v\circ\pi:\pi^{-1}(V)\rightarrow\mathbb{R}$ solves
\begin{equation}
\Delta_{M}u+\left(  a\circ\pi\right)  u=f(u). \label{eqM}%
\end{equation}
Conversely, if $\pi:\pi^{-1}(V)\rightarrow V$ is surjective and
$v:V\rightarrow\mathbb{R}$ is such that $u:=v\circ\pi:\pi^{-1}(V)\rightarrow
\mathbb{R}$ solves \emph{(\ref{eqM})} then $v$ solves \emph{(\ref{eqN})}.
\end{corollary}

\begin{proof}
This follows easily from Proposition \ref{prophm}.
\end{proof}

Next we give some examples of harmonic morphisms.

\begin{proposition}
Let $\pi:M\rightarrow N$ be a Riemannian submersion. Then $\pi$ is a harmonic
map iff each fiber $\pi^{-1}(y)$ is a minimal submanifold of $M$ (i.e. the
mean curvature of $\pi^{-1}(y)$ in $M$ is zero).
\end{proposition}

\begin{proof}
See \cite[(1.12)]{er}.
\end{proof}

Consequently, harmonic morphisms with constant non-zero dilation are simply
Riemannian submersions with minimal fibres, up to scale. Some interesting
examples are the Hopf fibrations.

\begin{example}
\label{exHfib}The Hopf fibrations $\mathbb{S}^{n}\rightarrow\mathbb{R}P^{n}$,
$\mathbb{S}^{2n+1}\rightarrow\mathbb{C}P^{n}$, $\mathbb{S}^{4n+3}%
\rightarrow\mathbb{H}P^{n}$ and $\mathbb{S}^{15}\rightarrow\mathbb{S}^{8}$ are
Riemannian submersions (up to scale) with totally geodesic, and so minimal,
fibres, see \cite[Examples 2.4.14-17]{bw}.
\end{example}

\begin{example}
The Hopf fibration $\mathbb{S}^{2n+1}\rightarrow\mathbb{C}P^{n}$ factors
through the double covering $\mathbb{S}^{2n+1}\rightarrow\mathbb{R}P^{2n+1}$
to give a Riemannian submersion $\mathbb{R}P^{2n+1}\rightarrow\mathbb{C}P^{n}$
with totally geodesic fibres. Similarly, one obtains a Riemannian submersion
$\mathbb{C}P^{2n+1}\rightarrow\mathbb{H}P^{n}$ with totally geodesic fibres.
\end{example}

The main example for our purposes is the following one.

\begin{example}
\label{exHmap}The Hopf maps $\pi:\mathbb{R}^{N}=\mathbb{K}\times
\mathbb{K}\rightarrow\mathbb{K}\times\mathbb{R}=\mathbb{R}^{\dim\mathbb{K}+1}$
given by%
\[
\pi(z_{1},z_{2}):=(2\overline{z_{1}}z_{2},\left\vert z_{1}\right\vert
^{2}-\left\vert z_{2}\right\vert ^{2}),
\]
with $\mathbb{K=R}$, $\mathbb{C}$, $\mathbb{H}$, or $\mathbb{O}$ respectively,
are harmonic morphisms \cite[Corollary 5.3.3]{bw} with dilation $\lambda
(x,y)=\sqrt{2(\left\vert x\right\vert ^{2}+\left\vert y\right\vert ^{2})}.$
Their restrictions to the unit sphere are the Hopf fibrations of \emph{Example
\ref{exHfib}} with $n=1.$ A simple computation shows that $\left\vert
\pi(x,y)\right\vert =\left\vert x\right\vert ^{2}+\left\vert y\right\vert
^{2}.$ Hence, $\lambda^{2}(x,y)=2\left\vert \pi(x,y)\right\vert .$
\end{example}

\bigskip

\noindent\textbf{Proof of Proposition \ref{propHopf}.}\emph{\qquad}Apply
Corollary \ref{correduction} to Example \ref{exHmap}. \qed\noindent

\section{Existence}

\label{sec:existence}Proposition \ref{propHopf} suggests considering the
problem%
\begin{equation}
\left\{
\begin{array}
[c]{ll}%
-\Delta v=K(x)\left\vert v\right\vert ^{2^{\ast}-2}v & \text{in }U,\\
\hspace{0.6cm}v=0 & \text{on }U,
\end{array}
\right.  \tag{$\wp_U^\ast$}\label{prob4}%
\end{equation}
where $U$ is a bounded smooth domain in $\mathbb{R}^{M}$, $K\in\mathcal{C}%
^{0}(\mathbb{R}^{M})$ is strictly positive on $\overline{U}$ and $2^{\ast
}:=\frac{2M}{M-2}$ is Sobolev's critical exponent. We assume that $U$ and $K$
are $G$-invariant for some closed subgroup $G$ of $O(M)$. Then, the principle
of symmetric criticality \cite{p}\ asserts that the $G$-invariant solutions of
problem (\ref{prob4}) are the critical points of the restriction of the
functional%
\[
J(v):=\frac{1}{2}\int_{U}\left\vert \nabla v\right\vert ^{2}-\frac{1}{2^{\ast
}}\int_{U}K(x)\left\vert v\right\vert ^{2^{\ast}}%
\]
to the space of $G$-invariant functions
\[
H_{0}^{1}(U)^{G}:=\{v\in H_{0}^{1}(U):v(gx)=v(x)\text{ \ for all }g\in
G,\text{ }x\in U\}.
\]
We shall say that $J$ satisfies the Palais-Smale condition $(PS)_{c}^{G}$ in
$H_{0}^{1}(U)$ if every sequence $(v_{n})$ such that
\[
v_{n}\in H_{0}^{1}(U)^{G},\qquad J(v_{n})\rightarrow c,\qquad\nabla
J(v_{n})\rightarrow0,
\]
contains a convergent subsequence. Let $S$ be the best Sobolev constant for
the embedding $D^{1,2}(\mathbb{R}^{M})\hookrightarrow L^{2^{\ast}}%
(\mathbb{R}^{M}).$ The following result was proved in \cite[Corollary 2]{c}.

\begin{proposition}
\label{propPS}$J$ satisfies condition $(PS)_{c}^{G}$ in $H_{0}^{1}(U)$ for
every%
\[
c<\left(  \min_{x\in\overline{U}}\frac{\#Gx}{K(x)^{\frac{M-2}{2}}}\right)
\frac{1}{M}S^{M/2}.
\]
In particular, if $\#Gx=\infty$ for all $x\in\overline{U},$ then $J$ satisfies
condition $(PS)_{c}^{G}$ in $H_{0}^{1}(U)$ for every $c\in\mathbb{R}$.
\end{proposition}

Fix a closed subgroup $\Gamma$ of $O(M)$ and a nonempty $\Gamma$-invariant
bounded smooth domain $D$ in $\mathbb{R}^{M}$ such that $\#\Gamma x=\infty
\ $for all $x\in D.$ Then, the following holds.

\begin{theorem}
\label{thmcfK}Assume that $K$ is $\Gamma$-invariant. Then, there exists an
increasing sequence $(\ell_{m})$ of positive real numbers, depending only on
$\Gamma,$ $D$ and $K$, with the following property: if $U$ contains $D$ and if
it is invariant under the action of a closed subgroup $G$ of $\Gamma$ for
which%
\[
\min_{x\in\overline{U}}\frac{\#Gx}{K(x)^{\frac{M-2}{2}}}>\ell_{m}%
\]
holds, then problem \emph{(\ref{prob4})} has at least $m$ pairs of\ $G$%
-invariant solutions $\pm v_{1},\ldots,\pm v_{m}$ such that $v_{1}$ is
positive, $v_{2},\ldots,v_{m}$ change sign, and
\[
\int_{U}\left\vert \nabla v_{k}\right\vert ^{2}\leq\ell_{k}S^{M/2}\text{\qquad
for every }k=1,\ldots,m.
\]

\end{theorem}

\begin{proof}
For $K=1$ this was proved in \cite[Theorem 1]{cf}. The proof for general $K$
goes through with minor modifications. We sketch it here for the reader's
convenience. Let $\mathcal{P}_{1}(D)$ be the set of all nonempty $\Gamma
$-invariant bounded smooth domains contained in $D,$ and define%
\[
\mathcal{P}_{k}(D):=\{(D_{1},\mathcal{\ldots},D_{k}):D_{i}\in\mathcal{P}%
_{1}(D)\text{, \ }D_{i}\cap D_{j}=\emptyset\text{ if }i\neq j\}.
\]
Note that $\mathcal{P}_{k}(D)\neq\emptyset$ for every $k\in\mathbb{N}.$ Since
$\#\Gamma x=\infty$ for all $x\in D_{i},$ Proposition \ref{propPS} allows to
apply the mountain pass theorem \cite{ar} to obtain a nontrivial least energy
$\Gamma$-invariant solution $\omega_{D_{i}}\ $to problem $(\wp_{D_{i}}^{\ast
}).$ Extending $\omega_{D_{i}}$ by zero outside $D_{i}$ we have that
$\omega_{D_{i}}\in H_{0}^{1}(\Omega)^{G}$ and%
\begin{equation}
J(\omega_{D_{i}})=\max_{t\geq0}J(t\omega_{D_{i}}). \label{mountainpass}%
\end{equation}
We define%
\[
c_{k}:=\inf\left\{
{\textstyle\sum\limits_{i=1}^{k}}
J(\omega_{D_{i}}):(D_{1},\mathcal{\ldots},D_{k})\in\mathcal{P}_{k}(D)\right\}
\text{\qquad and\qquad}\ell_{k}:=\left(  \frac{1}{M}S^{M/2}\right)  ^{-1}%
c_{k}.
\]
Note that $c_{1}=J(\omega_{D})>0$ and that $J(\omega_{D_{i}})\geq c_{1}.$
Therefore,%
\[
c_{k-1}+c_{1}\leq%
{\textstyle\sum\limits_{i=1}^{k}}
J(\omega_{D_{i}})
\]
for every $(D_{1},\mathcal{\ldots},D_{k})\in\mathcal{P}_{k}(D)$, $k\geq2.$ It
follows that
\[
c_{k-1}+c_{1}\leq c_{k}\text{\qquad and\qquad}\ell_{k-1}+\ell_{1}\leq\ell
_{k}.
\]
Let $m\in\mathbb{N}$ and let $\Omega$ be a bounded smooth domain containing
$D,$ which is invariant under the action of a closed subgroup $G$ of $\Gamma$
for which%
\begin{equation}
\min_{x\in\overline{U}}\frac{\#Gx}{K(x)^{\frac{M-2}{2}}}>\ell_{m}
\label{hypothesis}%
\end{equation}
holds. Given $\varepsilon\in(0,c_{1})$ with $c_{m}+\varepsilon<\left(
\min_{x\in\overline{U}}\frac{\#Gx}{K(x)^{\frac{M-2}{2}}}\right)  \frac{1}%
{M}S^{M/2},$ we choose $(D_{1},\mathcal{\ldots},D_{m})\in\mathcal{P}_{m}(D)$
such that%
\[
c_{m}\leq%
{\textstyle\sum\limits_{i=1}^{m}}
J(\omega_{D_{i}})<c_{m}+\varepsilon.
\]
For each $k=1,\ldots,m,$ let $W_{k}$ be the subspace of $H_{0}^{1}(\Omega
)^{G}$ generated by $\{\omega_{D_{1}},\ldots,\omega_{D_{k}}\}$ and
$d_{k}:=\sup_{W_{k}}J.$ Then, $\dim W_{k}=k$ and identity (\ref{mountainpass})
implies that
\[
d_{k}=\sup_{W_{k}}J\leq%
{\textstyle\sum\limits_{i=1}^{k}}
J(\omega_{D_{i}})<\left(  \min_{x\in\overline{U}}\frac{\#Gx}{K(x)^{\frac
{M-2}{2}}}\right)  \frac{1}{M}S^{M/2}.
\]
Then, by Proposition \ref{propPS}, $J$ satisfies $(PS)_{c}^{G}$ in $H_{0}%
^{1}(\Omega)$ for all $c\leq d_{k},$ so the mountain pass theorem \cite{ar}
yields a positive critical point $v_{1}\in H_{0}^{1}(\Omega)^{G}$ of $J$ such
that $J(v_{1})\leq d_{1},$ and Theorem 3.7 in \cite{cp}, conveniently adapted
to the functional we are considering here, yields $m-1$ pairs of sign changing
critical points $\pm v_{2},\ldots,\pm v_{m}\in H_{0}^{1}(\Omega)^{G}$ such
that%
\[
J(v_{k})\leq d_{k}\text{\qquad for every }k=2,\ldots,m.
\]
The proof that $v_{k}$ may be chosen so that $J(u_{k})\leq c_{k}$ for every
\ $k=1,\ldots,m,$ follows just as in \cite{cf}.
\end{proof}

\bigskip

\noindent\textbf{Proof of Theorem \ref{thmmain2}.}\emph{\qquad}This follows
from Theorem \ref{thmcfK} and Proposition \ref{propHopf}. \qed\noindent

\section{Nonexistence}

\label{sec:nonexistence}Fix $k_{1},\ldots,k_{m}\in\mathbb{N}\cup\{0\}$ with
$k:=k_{1}+\cdots+k_{m}\leq N-3$ and a bounded smooth domain $\Theta$ in
$\mathbb{R}^{N-k}$ with $\overline{\Theta}\subset\left(  0,\infty\right)
^{m}\times\mathbb{R}^{N-k-m}$. Set
\begin{equation}
\Omega:=\{(y^{1},\ldots,y^{m},z)\in\mathbb{R}^{k_{1}+1}\times\cdots
\times\mathbb{R}^{k_{m}+1}\times\mathbb{R}^{N-k-m}:\left(  \left\vert
y^{1}\right\vert ,\ldots,\left\vert y^{m}\right\vert ,z\right)  \in\Theta\}.
\label{omega}%
\end{equation}
Let $G:=O(k_{1}+1)\times\cdots\times O(k_{m}+1).$ We think of $G$ as a
subgroup of $O(N)$ acting on $\mathbb{R}^{k_{1}+1}\times\cdots\times
\mathbb{R}^{k_{m}+1}\times\mathbb{R}^{N-k-m}$ in the obvious way, i.e.
\begin{equation}
(g_{1},\ldots,g_{m})(y^{1},\ldots,y^{m},z):=(g_{1}y^{1},\ldots,g_{m}y^{m},z)
\label{Gaction}%
\end{equation}
for $g_{i}\in O(k_{i}+1),$ $y^{i}\in\mathbb{R}^{k_{i}+1},$ $z\in
\mathbb{R}^{N-k-m}.$ Then $\Omega$ is $G$-invariant. For $K\in\mathcal{C}%
^{0}(\overline{\Omega})$ we consider the problem%
\begin{equation}
\left\{
\begin{array}
[c]{ll}%
-\Delta u=K(x)\left\vert u\right\vert ^{p-2}u & \text{in }\Omega,\\
\hspace{0.6cm}u=0 & \text{on }\partial\Omega.
\end{array}
\right.  \label{probK2}%
\end{equation}

\begin{proposition}
\label{propexK}If $K$ is positive and $G$-invariant in $\overline{\Omega}$ and
$0\leq k\leq N-3,$ then problem \emph{(\ref{probK2})} has infinitely many
$G$-invariant solutions for $p\in(2,2_{N,k}^{\ast}).$
\end{proposition}

\begin{proof}
A $G$-invariant function $u(y^{1},\ldots,y^{m},z)=v(\left\vert y^{1}%
\right\vert ,\ldots,\left\vert y^{m}\right\vert ,z)$ solves problem
(\ref{probK2}) if and only if $v$ solves
\[
-\Delta v-\sum_{i=1}^{m}\frac{k_{i}}{x_{i}}\frac{\partial v}{\partial x_{i}%
}=K(x)|v|^{p-2}v\quad\text{in}\ \Theta,\qquad v=0\quad\text{on}\ \partial
\Theta.
\]
This problem can be rewritten as%
\begin{equation}
-\text{div}(a(x)\nabla v)=Q(x)|v|^{p-2}v\quad\text{in}\ \Theta,\qquad
v=0\quad\text{on}\ \partial\Theta, \label{eqdiv}%
\end{equation}
where $a(x_{1},\ldots,x_{N-k}):=x_{1}^{k_{1}}\cdots x_{m}^{k_{m}}$ and
$Q(x):=a(x)K(x).$ Note that both $a$ and $Q$ are continuous and strictly
positive in $\overline{\Theta}.$ Hence, the norms%
\[
\left\Vert v\right\Vert _{a}:=\left(  \int_{\Theta}a(x)\left\vert \nabla
v\right\vert ^{2}\right)  ^{1/2}\text{\qquad and\qquad}\left\vert v\right\vert
_{Q,p}:=\left(  \int_{\Theta}Q(x)\left\vert v\right\vert ^{p}\right)  ^{1/p}%
\]
are equivalent to those of $H_{0}^{1}(\Theta)$ and $L^{p}(\Theta)$
respectively. Since $H_{0}^{1}(\Theta)$ is compactly embedded in $L^{p}%
(\Theta)$ for $p<2_{N-k}^{\ast},$ the functional%
\[
J(v):=\frac{1}{2}\left\Vert v\right\Vert _{a}^{2}-\frac{1}{p}\left\vert
v\right\vert _{Q,p}^{p},\text{\qquad}v\in H_{0}^{1}(\Theta),
\]
satisfies the Palais-Smale condition. It clearly satisfies all other
hypotheses of the symmetric mountain pass theorem \cite{ar}. Hence, it has an
unbounded sequence of critical values. The critical values of $J$ are the
solutions of (\ref{eqdiv}).
\end{proof}

Next, fix $\tau_{1},\ldots,\tau_{m}\in(0,\infty),$ and let $\varphi_{i}$ be
the solution to the problem%
\[
\left\{
\begin{array}
[c]{ll}%
\varphi_{i}^{\prime}(t)t+(k_{i}+1)\varphi_{i}(t)=1, & t\in(0,\infty),\\
\varphi_{i}(\tau_{i})=0. &
\end{array}
\right.
\]
Explicitly, $\varphi_{i}(t)=\frac{1}{k_{i}+1}\left[  1-(\frac{\tau_{i}}%
{t})^{k_{i}+1}\right]  .$ Note that $\varphi_{i}$ is strictly increasing in
$(0,\infty).$ For $y^{i}\neq0$ we define%
\begin{equation}
\chi(y^{1},\ldots,y^{m},z):=(\varphi_{1}(\left\vert y^{1}\right\vert
)y^{1},\ldots,\varphi_{m}(\left\vert y^{m}\right\vert )y^{m},z). \label{vf}%
\end{equation}

\begin{lemma}
\label{lemvf}$\chi$ has the following properties:

\begin{enumerate}
\item[(a)] \emph{div}$\chi=N-k,$

\item[(b)] $\left\langle \mathrm{d}\chi(y^{1},\ldots,y^{m},z)\left[
\xi\right]  ,\xi\right\rangle \leq\max\left\{  1-k_{1}\varphi_{1}(\left\vert
y^{1}\right\vert ),\ldots,1-k_{m}\varphi_{m}(\left\vert y^{m}\right\vert
),1\right\}  \left\vert \xi\right\vert ^{2}$ \ for every $y^{i}\in
\mathbb{R}^{k_{i}+1}\smallsetminus\{0\},$ $z\in\mathbb{R}^{N-k-m},$ $\xi
\in\mathbb{R}^{N}.$
\end{enumerate}
\end{lemma}

\begin{proof}
(a) Write $y^{i}=(y_{1}^{i},\ldots,y_{k_{i}+1}^{i}).$ Then,
\[
\text{div}\chi(y^{1},\ldots,y^{m},z)=%
{\textstyle\sum\limits_{i=1}^{m}}
\left[
{\textstyle\sum\limits_{j=1}^{k_{i}+1}}
\varphi_{i}^{\prime}(\left\vert y^{i}\right\vert )\frac{(y_{j}^{i})^{2}%
}{\left\vert y^{i}\right\vert }+(k_{i}+1)\varphi_{i}(\left\vert y^{i}%
\right\vert )\right]  +N-k-m=N-k.
\]
(b) $\chi$ is $G$-equivariant for the $G$-action defined in (\ref{Gaction}),
that is,
\[
\chi(gy,z)=g\chi(y,z)
\]
for every $g\in G,$ $y=(y^{1},\ldots,y^{m}),$ $y^{i}\in\mathbb{R}^{k_{i}%
+1}\smallsetminus\{0\},$ $z\in\mathbb{R}^{N-k-m}.$ Therefore, $g\circ
\mathrm{d}\chi(y,z)=\mathrm{d}\chi(gy,z)\circ g$ and, hence,%
\[
\left\langle \mathrm{d}\chi\left(  y,z\right)  \left[  \xi\right]
,\xi\right\rangle =\left\langle g\left(  \mathrm{d}\chi\left(  y,z\right)
\left[  \xi\right]  \right)  ,g\xi\right\rangle =\left\langle \mathrm{d}%
\chi\left(  gy,z\right)  \left[  g\xi\right]  ,g\xi\right\rangle
\]
for all $\xi\in\mathbb{R}^{N}.$ Thus, it suffices to show that the inequality
(b) holds for $y^{i}=(y_{1}^{i},0,\ldots,0)$ with $y_{1}^{i}>0.$ Set $\chi
_{i}(y^{i}):=\varphi_{i}(\left\vert y^{i}\right\vert )y^{i}.$ A
straightforward computation shows that, for such $y^{i},$ $\mathrm{d}\chi
_{i}(y^{i})$ is a diagonal matrix whose diagonal entries are $a_{11}%
=1-k_{i}\varphi_{i}(y_{1}^{i})$ and $a_{jj}=\varphi_{i}(y_{1}^{i})$ for
$j=2,\ldots,k_{i}+1.$ Since $\varphi_{i}(t)<\frac{1}{k_{i}+1}$ for all
$t\in(0,\infty),$ (b) follows.\bigskip
\end{proof}

\noindent\textbf{Proof of Theorem \ref{thmmain1}.}\emph{\qquad}The variational
identity (4) in Pucci and Serrin's paper \cite{ps} implies that, if
$u\in\mathcal{C}^{2}(\Omega)\cap\mathcal{C}^{1}(\overline{\Omega})$ is a
solution of (\ref{probK}) and $\chi\in\mathcal{C}^{1}(\overline{\Omega
},\mathbb{R}^{N}),$ then%
\begin{align}
\frac{1}{2}\int_{\partial\Omega}\left\vert \nabla u\right\vert ^{2}%
\left\langle \chi,\nu_{\Omega}\right\rangle d\sigma &  =\int_{\Omega}\left(
\text{div}\chi\right)  \left[  \frac{1}{p}K\left\vert u\right\vert ^{p}%
-\frac{1}{2}\left\vert \nabla u\right\vert ^{2}\right]  dx\nonumber\\
&  +\frac{1}{p}\int_{\Omega}\left\vert u\right\vert ^{p}\left\langle
\chi,\nabla K\right\rangle dx+\int_{\Omega}\left\langle \mathrm{d}\chi\left[
\nabla u\right]  ,\nabla u\right\rangle dx \label{idPS}%
\end{align}
where $\nu_{\Omega}$ is the outward pointing unit normal to $\partial\Omega$.
Take $\chi$ to be the vector field defined in (\ref{vf}) for $m=1,$ $0\leq
k\leq N-3$ and $\tau_{1}=t_{0}$ as in Definition \ref{**}, that is,%
\[
\chi(y,z):=(\varphi(\left\vert y\right\vert )y,z),\text{\qquad}(y,z)\in\left(
\mathbb{R}^{k+1}\smallsetminus\{0\}\right)  \times\mathbb{R}^{N-k-1}%
\]
with $\varphi(t)=\frac{1}{k+1}\left[  1-(\frac{t_{0}}{t})^{k+1}\right]  .$
Then, by Lemma \ref{lemvf},
\begin{equation}
\text{div}\chi=N-k. \label{fact2}%
\end{equation}
Note that, since $\varphi(t)\geq0$ for $t\in(t_{0},\infty)$ and $\left\vert
y\right\vert >t_{0}$ if $(y,z)\in\Omega,$ we have that
\begin{equation}
\left\langle \chi(y,z),\nabla K(y,z)\right\rangle =\varphi(\left\vert
y\right\vert )\left\langle y,\partial_{y}K(y,z)\right\rangle +\left\langle
z,\partial_{z}K(y,z)\right\rangle \leq0\text{\quad}\forall(y,z)\in\Omega.
\label{fact1}%
\end{equation}
Moreover, since $1-k\varphi(t)<1$ for $t\in(t_{0},\infty)$, Lemma
\ref{lemvf}\ yields
\begin{equation}
\left\langle \mathrm{d}\chi\left(  x\right)  \left[  \xi\right]
,\xi\right\rangle \leq\left\vert \xi\right\vert ^{2}\qquad\forall x\in
\Omega,\text{ }\xi\in\mathbb{R}^{N}. \label{fact3}%
\end{equation}
We claim that \
\begin{equation}
\left\langle \chi(x),\nu_{\Omega}(x)\right\rangle >0\qquad\forall x\in
\partial\Omega\smallsetminus\left\{  g\xi_{0},g\xi_{1}:g\in O(k+1)\right\}  .
\label{fact4}%
\end{equation}
Since $\Omega$ is $O(k+1)$-invariant, $\nu_{\Omega}$ is $O(k+1)$-equivariant.
Thus, it suffices to show that%
\begin{equation}
\left\langle (\varphi(t)t,z),\nu_{\Theta}(t,z)\right\rangle >0\qquad\text{for
all }(t,z)\in\partial\Theta\smallsetminus\left\{  \xi_{0},\xi_{1}\right\}  ,
\label{fact4a}%
\end{equation}
where $\nu_{\Theta}(t,z)$ is the outward pointing unit normal to
$\partial\Theta$ at $(t,z)$ which we write as $\nu_{\Theta}(t,z)=\left(
\nu_{1}(t,z),\nu_{2}(t,z)\right)  \in\mathbb{R}\times\mathbb{R}^{N-k-1}.$ Let
$(t,z)\in\partial\Theta.$ Since $\Theta$ is doubly starshaped we have that%
\[
(t-t_{i})\nu_{1}(t,z)+\left\langle z,\nu_{2}(t,z)\right\rangle >0\text{\qquad
if }(t,z)\neq(t_{i},0),\text{ for }i=0,1,
\]
with $t_{0},t_{1}$ as in Definition \ref{**}. Therefore,%
\[
\left\langle (\varphi(t)t,z),\nu_{\Theta}(t,z)\right\rangle =\varphi
(t)t\nu_{1}(t,z)+\left\langle z,\nu_{2}(t,z)\right\rangle >(\varphi
(t)t-t+t_{i})\nu_{1}(t,z).
\]
Set $\psi(t):=\varphi(t)t-t.$ Note that $\psi^{\prime}(t)=-k\varphi(t)<0$ if
$t>t_{0}.$ So, since $t\in(t_{0},t_{1})$ for every $(t,z)\in$ $\Theta,$ we
have that%
\[
\varphi(t_{1})t_{1}-t_{1}=\psi(t_{1})\leq\psi(t)\leq\psi(t_{0})=-t_{0}%
\text{\qquad}\forall(t,z)\in\partial\Theta.
\]
If $\nu_{1}(t,z)\leq0$, then%
\[
\left\langle (\varphi(t)t,z),\nu_{\Theta}(t,z)\right\rangle >(\psi
(t)+t_{0})\nu_{1}(t,z)\geq0
\]
and if $\nu_{1}(t,z)\geq0$, then
\[
\left\langle (\varphi(t)t,z),\nu_{\Theta}(t,z)\right\rangle >(\psi
(t)+t_{1})\nu_{1}(t,z)\geq\varphi(t_{1})t_{1}\nu_{1}(t,z)\geq0.
\]
This proves (\ref{fact4a}).

\noindent Combining properties (\ref{fact2}), (\ref{fact1}), (\ref{fact3}) and
(\ref{fact4}) with identity (\ref{idPS}) gives%
\begin{align*}
0  &  <\int_{\Omega}\left(  \text{div}\chi\right)  \left[  \frac{1}%
{p}K\left\vert u\right\vert ^{p}-\frac{1}{2}\left\vert \nabla u\right\vert
^{2}\right]  dx+\int_{\Omega}\left\vert \nabla u\right\vert ^{2}dx\\
&  =(N-k)\left(  \frac{1}{p}-\frac{1}{2}+\frac{1}{N-k}\right)  \int_{\Omega
}\left\vert \nabla u\right\vert ^{2}dx
\end{align*}
which implies that $p<2_{N,k}^{\ast}$ if $u\neq0.$

\noindent Proposition \ref{propexK} yields infinitely many solutions for
$p<2_{N,k}^{\ast}$. \qed\noindent

\bigskip

\noindent\textbf{Proof of Theorem \ref{thmmain4}.}\qquad Choose $\alpha
\in(1,\frac{N-k}{2})$ with $2_{N,k}^{\ast}+\varepsilon\geq\frac{2(N-k)}%
{N-k-2\alpha}.$ Fix $\tau_{1},\ldots,\tau_{m}\in(0,\infty)$ and, for the given
$k_{1},\ldots,k_{m},$ define $\chi$ as in (\ref{vf})$.$ Let $0<\varrho
<\tau_{i}$ be defined by%
\[
\max\left\{  1-k_{1}\varphi_{1}(\tau_{1}-\varrho),\ldots,1-k_{m}\varphi
_{m}(\tau_{m}-\varrho)\right\}  =\alpha,
\]
$\Theta:=B_{\varrho}^{N-k}(\tau)$ be the ball of radius $\varrho$ centered at
$\tau=(\tau_{1},\ldots,\tau_{m},0)$ in $\mathbb{R}^{m}\times\mathbb{R}%
^{N-k-m}$ and $\Omega$ be defined as in (\ref{omega}). Then $\Omega$ has the
homotopy type of $\mathbb{S}^{k_{1}}\times\cdots\times\mathbb{S}^{k_{m}}.$
Moreover, Lemma \ref{lemvf}\ asserts that
\begin{equation}
\text{div}\chi=N-k\text{\qquad and\qquad}\left\langle \mathrm{d}\chi\left(
x\right)  \left[  \xi\right]  ,\xi\right\rangle \leq\alpha\left\vert
\xi\right\vert ^{2}\quad\forall x\in\Omega,\text{ }\xi\in\mathbb{R}^{N}.
\label{fact5}%
\end{equation}
Since $\varphi_{i}(t)<0$ if $t<\tau_{i}$ and $\varphi_{i}(t)>0$ if $t>\tau
_{i}$ we have that, for all but a finite number of points $(x,z)\in
\partial\Theta,$
\[
\left\langle (\varphi_{1}(x_{1})x_{1},\ldots,\varphi_{m}(x_{m})x_{m}%
,z),\,\nu_{\Theta}(t,z)\right\rangle =%
{\textstyle\sum\limits_{i=1}^{m}}
\varphi_{i}(x_{i})x_{i}(x_{i}-\tau_{i})+\left\vert z\right\vert ^{2}>0.
\]
Hence,
\begin{equation}
\left\langle \chi,\nu_{\Omega}\right\rangle >0\qquad\text{a.e. on }%
\partial\Omega. \label{fact6}%
\end{equation}
Combining properties (\ref{fact5}) and (\ref{fact6}) with identity
(\ref{idPS}) for $K=1$ we obtain%
\begin{align*}
0  &  <\int_{\Omega}\left(  \text{div}\chi\right)  \left[  \frac{1}%
{p}\left\vert u\right\vert ^{p}-\frac{1}{2}\left\vert \nabla u\right\vert
^{2}\right]  dx+\alpha\int_{\Omega}\left\vert \nabla u\right\vert ^{2}dx\\
&  =(N-k)\left(  \frac{1}{p}-\frac{1}{2}+\frac{\alpha}{N-k}\right)
\int_{\Omega}\left\vert \nabla u\right\vert ^{2}dx
\end{align*}
which implies that $p<\frac{2(N-k)}{N-k-2\alpha}\leq2_{N,k}^{\ast}%
+\varepsilon$ if $u\neq0.$ Consequently, problem (\ref{prob}) does not have a
nontrivial solution in $\Omega$ for $p\geq2_{N,k}^{\ast}+\varepsilon$, whereas
Proposition \ref{propexK} yields infinitely many solutions for $p<2_{N,k}%
^{\ast}$. \qed\noindent

\bigskip

\noindent\textbf{Proof of Theorem \ref{thmmain3}.}\emph{\qquad}For $0\leq
n\leq\dim\mathbb{K}-2,$ let $\Theta_{n}$ be a bounded smooth domain in
$\mathbb{R}^{\dim\mathbb{K}-n+1}$ with $\overline{\Theta_{n}}\subset\left(
0,\infty\right)  \times\mathbb{R}^{\dim\mathbb{K}-n},$ which is doubly
starshaped with respect to $\mathbb{R}\times\{0\}$. Define $U_{0}:=\Theta_{0}$
and
\[
U_{n}:=\{(y,z)\in\mathbb{R}^{n+1}\times\mathbb{R}^{\dim\mathbb{K}-n}:\left(
\left\vert y\right\vert ,z\right)  \in\Theta_{n}\}\subset\mathbb{R}%
^{\dim\mathbb{K}+1}%
\]
if $n\geq1.$ Theorem \ref{thmmain1} asserts that problem
\[
-\Delta v=\frac{1}{2\left\vert x\right\vert }\left\vert v\right\vert
^{p-2}v\quad\text{in }U_{n},\qquad v=0\quad\text{on }\partial U_{n},
\]
has infinitely many solutions for $p<2_{\dim\mathbb{K}+1,n}^{\ast}$ and no
nontrivial solutions for $p\geq2_{\dim\mathbb{K}+1,n}^{\ast}.$ Hence,
Proposition \ref{propHopf} implies that problem%
\[
-\Delta u=\left\vert u\right\vert ^{p-2}u\quad\text{in }\Omega_{n}:=\pi
^{-1}(U_{n}),\qquad u=0\quad\text{on }\partial\Omega_{n},
\]
has infinitely many $\mathbb{S}_{\mathbb{K}}$-invariant solutions if
$p<2_{N,k}^{\ast}$ and does not have a nontrivial $\mathbb{S}_{\mathbb{K}}%
$-invariant solution if $p\geq2_{N,k}^{\ast},$ where $k:=\dim\mathbb{K}-1+n.$

\noindent Finally, since the restriction of the Hopf map $\pi:\mathbb{R}%
^{N}\smallsetminus\{0\}\rightarrow\mathbb{R}^{\dim\mathbb{K}+1}\smallsetminus
\{0\}$ is a fibration and $U_{n}$ is contractible in $\mathbb{R}%
^{\dim\mathbb{K}+1}\smallsetminus\{0\},$ the domain $\Omega_{n}$ is fiber
homotopy equivalent to $\mathbb{S}_{\mathbb{K}}\times U_{n}$ \cite[Chap.2,
Sec.8, Theorem 14]{s}. Hence, it has the homotopy type of $\mathbb{S}%
_{\mathbb{K}}\times\mathbb{S}^{n}$ if $n\geq1$ and of $\mathbb{S}_{\mathbb{K}%
}$ if $n=0.$ \qed

\end{document}